\documentclass[11pt]{article}

\usepackage[margin=1in]{geometry}
\usepackage{amsmath,amssymb,amsthm,mathtools}
\usepackage{enumitem}
\usepackage{microtype}
\usepackage[hidelinks]{hyperref}
\usepackage{comment}
\usepackage{mathrsfs}
\newcommand{\Sch}{\mathscr S}

\newtheorem{theorem}{Theorem}[section]
\newtheorem{proposition}[theorem]{Proposition}
\newtheorem{lemma}[theorem]{Lemma}
\newtheorem{corollary}[theorem]{Corollary}

\makeatletter
\NewCommandCopy\@@pmod\pmod
\DeclareRobustCommand{\pmod}{\@ifstar\@pmods\@@pmod}
\def\@pmods#1{\mkern4mu({\operator@font mod}\mkern 6mu#1)}
\makeatother

\newcommand{\R}{\mathbb R}
\newcommand{\Z}{\mathbb Z}
\newcommand{\T}{\mathbb T}

\newcommand{\eps}{\varepsilon}
\newcommand{\cB}{\mathcal B}
\newcommand{\cF}{\mathcal F}
\newcommand{\cM}{\mathfrak M}
\newcommand{\supp}{\operatorname{supp}}
\newcommand{\C}{\mathbb C}

\title{The van der Corput property for sums of two squares}
\author{
Steve Fan\thanks{Steve.Fan@uga.edu}~$^{,1}$ \and
Andrew Lott\thanks{Andrew.Lott@uga.edu}~$^{,1,2}$
}
\date{June 23, 2026}

\begin{document}
\maketitle

\begin{center}
University of Georgia$^{1}$ \\
HUN-REN Alfr\'ed R\'enyi Institute of Mathematics (Erd\H{o}s Center)$^{2}$
\end{center}

\begin{abstract}
Let $S_N=\{1\le d\le N:d=x^2+y^2\text{ for some }x,y\in\mathbb Z\}.$
We prove a power-saving form of the van der Corput property for \(S_N\). As a consequence, we obtain a strong Sárközy-type result: if \(A\subseteq [N]\) has no nonzero difference equal to a sum of two squares, then $|A|\ll_\varepsilon N^{7/8+\varepsilon}$ for every $\eps>0$, improving upon an earlier quasipolynomial bound due to Rice. The shape of this bound is optimal, as a construction of Younis yields a set \(A\subseteq [N]\) with $|A|\gg N^{1/2}$
such that \((A-A)\cap S_N=\emptyset\). 
\end{abstract}

\tableofcontents

\section{Introduction}

Let
\[
  S_N=\{1\le d\le N:d=x^2+y^2\text{ for some }x,y\in\mathbb Z\}.
\]
Our main result is the following quantitative form of the van der Corput property for the
set of sums of two squares.

\begin{theorem}[Power-saving van der Corput property]
\label{thm:power-saving-vdc}
For every \(\eps>0\) there is a constant \(C_\eps>0\) such that for every
positive integer \(N\), there are coefficients \(c_d\ge 0\), supported on
\(S_N\), satisfying
\[
  \sum_{d\in S_N}c_d=1
\]
and
\[
  \Re\sum_{d\in S_N}c_d e(d\theta)
  \ge -C_\eps N^{-1/8+\eps}
  \qquad(\theta\in\mathbb R/\mathbb Z).
\]
\end{theorem}

From this we immediately have a strong combinatorial result. 

\begin{corollary}[Sárközy-type theorem for sums of two squares]
\label{cor:sarkozy-two-squares}
For every \(\eps>0\) there is a constant \(C_\eps>0\) such that the following
holds.  If \(A\subseteq\{1,\dots,N\}\) satisfies
\[
  (A-A)\cap S_N=\varnothing,
\]
then
\[
  |A|\le C_\eps N^{7/8+\eps}.
\]
\end{corollary}

A quasipolynomial bound on $|A|$ had earlier been obtained directly by Rice \cite{RiceBinaryQuadraticForms}. In fact, his bound is applicable to subsets lacking nonzero differences equal to the values of any fixed nondegenerate integral binary quadratic form, and in particular, implies that if
\((A-A)\cap S_N=\emptyset\), then
\[
  |A|\ll N\exp(-c\sqrt{\log N}).
\]
In addition, the recent work of Green--Sawhney on the Furstenberg--Sárközy theorem for
squares \cite{GreenSawhneyFurstenbergSarkozy} gives
\[
  |A|\ll N\exp(-c\sqrt{\log N})
\]
whenever \(A\subseteq [N]\) has no nonzero square differences. Since every square is a sum of two squares, a set with
\((A-A)\cap S_N=\emptyset\) has no nonzero square differences.  Thus their
result also gives the same quasipolynomial bound 
in the present problem. It is impossible to improve the shape of the upper bound in Corollary
\ref{cor:sarkozy-two-squares} beyond a polynomial shape, as a construction of
Younis \cite{YounisLowerBounds} yields a set \(A\subseteq [N]\) with
\[
  |A|\gg N^{1/2}
\]
such that \((A-A)\cap S_N=\emptyset\). 

The fact that Theorem \ref{thm:power-saving-vdc} implies Corollary \ref{cor:sarkozy-two-squares} is short and standard, but we record it for the sake of completeness in the appendix. The idea behind this goes back to Kamae--Mendès France \cite{KamaeMendesFrance} and Ruzsa \cite{RuzsaConnections}. Montgomery \cite{MontgomeryTenLectures} gives a useful account of the method.
Matolcsi--Ruzsa develop related connections between difference sets and
positive exponential sums in
\cite{MatolcsiRuzsa2012GeneralProperties}; see
\cite{MatolcsiRuzsa2021CubicResidues} for an application to cubic residues in
$\Z/q\Z$.
A recent breakthrough of Green gives a striking demonstration of the strength
of this method: by proving a power-saving van der Corput theorem for the
shifted primes, he obtained the first power-saving bound for Sárközy's theorem for
differences of the form \(p-1\) with $p$ prime \cite{GreenShiftedPrimes}.

The main difficulty in this problem  is the definition of the coefficients
\(c_d\) in Theorem \ref{thm:power-saving-vdc}. Designing coefficients supported
on \(S_N\) satisfying a power-saving form of the van der Corput property is
difficult because of the lack of symmetry modulo \(q\). This is particularly
evident in the distribution of sums of two squares modulo \(2^m\), or modulo
\(p^m\) when \(p\equiv 3\pmod 4\). The corresponding complete sums may be
negative or purely imaginary, obstructing a lower bound for
\[
  \Re\sum_{d\in S_N}c_de(d\theta).
\]

The coefficients \(c_d\) come from a polynomial exponential sum in three
variables. In particular, unlike the shifted-prime weights in
\cite{GreenShiftedPrimes}, they do not involve truncated Ramanujan expansions.
The use of an additional variable in this construction is inspired by the work
of Slijep\v{c}evi\'c on squares and of
Nin\v{c}evi\'c--Slijep\v{c}evi\'c on odd polynomials
\cite{Slijepcevic2010Squares,NincevicSlijepcevic2013}. Even after the coefficient structure is identified, substantial difficulty
remains: we make a carefully designed choice of smooth weights and develop a
new major- and minor-arc analysis tailored to the resulting three-variable
exponential sum. In particular, our fixed-sector argument overcomes, in this
setting, the obstruction from complex complete sums: the complete sum and the
corresponding oscillatory integral factor may both be complex, and we control
the real part of their product by tracking their sectors. This mechanism is the
main novelty in our work.

\section{Notation}
\label{sec:notation}

We collect here the notation and conventions used throughout the paper.

\begin{itemize}[leftmargin=*]

\item For a positive integer \(N\), we write
\[
  [N]=\{1,2,\dots,N\}.
\]
If \(A\) is a set of integers, then
\[
  A-A=\{a-a':a,a'\in A\}.
\]
We write \(1_A\) for the indicator function of \(A\).

\item Let $\T=\R/\Z$. For \(t\in\R\), we put $e(t)=e^{2\pi i t}$ and write $\|t\|=\min_{n\in\Z}|t-n|$
for the distance from \(t\) to the nearest integer. If \(\theta\in\T\), then
\(\|\theta\|\) stands for the distance from \(\theta\) to \(0\) on \(\T\).

\item We write \(\supp f\) for the support of a function \(f\), and
\(C_c^\infty(\R^d)\) for the space of smooth compactly supported functions from
\(\R^d\) to \(\C\).  All integrals over \(\R^d\) are with respect to Lebesgue measure.

\item We write \(\Sch(\R^d)\) for the Schwartz class of complex-valued functions.  So \(f\in\Sch(\R^d)\) means that \(f\) is smooth and that for every pair of
multi-indices \(\alpha,\gamma\in\Z_{\ge0}^d\),
\[
  \sup_{x\in\R^d}
  \left|x^\alpha \partial^\gamma f(x)\right|
  <\infty.
\]
Equivalently, \(f\) and all of its derivatives decay faster than any power.

\item For \(f\in\Sch(\R^d)\), we have the Fourier transform:
\[
  \widehat f(\xi)
  =
  \int_{\R^d} f(x)e(-x\cdot \xi)\,dx,
  \qquad \xi\in\R^d.
\]

\item We also have the Poisson summation formula: if $f\in \Sch(\R^d)$, then 
\[
  \sum_{n\in\Z^d} f(n)
  =
  \sum_{\ell\in\Z^d} \widehat f(\ell).
\]
We will also use the scaled form
\[
  \sum_{n\in\Z^d} f\!\left(\frac nH\right)
  =
  H^d\sum_{\ell\in\Z^d}\widehat f(H\ell)
\]
for all $H>0$, which follows by applying Poisson summation to the function \(x\mapsto f(x/H)\).

\item We will also use Poisson summation after splitting into residue classes.
If \(f\in\Sch(\R^d)\), \(q\in\mathbb{N}\),
and \(r\in(\Z/q\Z)^d\), then
\begin{equation}\label{eq:f(n)n=r(mod q)}
\sum_{\substack{n\in\Z^d\\ n\equiv r\pmod* q}} f(n)
  =
  q^{-d}\sum_{\ell\in\Z^d}
  e\!\left(\frac{r\cdot \ell}{q}\right)
  \widehat f\!\left(\frac{\ell}{q}\right),  
\end{equation}
which follows by applying Poisson summation to the function $x \mapsto f(qx+r)$.
Moreover, if \(F\colon (\Z/q\Z)^d\to \C\) is any function, then multiplying both sides of \eqref{eq:f(n)n=r(mod q)} by $F(r)$ and summing over all residue classes $r\pmod* q$ give
\begin{equation}\label{eq:F(n mod q)}
\sum_{n\in\Z^d}F(n\pmod* q)f(n)
  =
  q^{-d}\sum_{\ell\in\Z^d}
  \left(
    \sum_{r\pmod* q}F(r)e\!\left(\frac{r\cdot \ell}{q}\right)
  \right)
  \widehat f\!\left(\frac{\ell}{q}\right).   
\end{equation}
Conversely, we recover \eqref{eq:f(n)n=r(mod q)} by taking $F(n)=1_{n\equiv r\pmod*{q}}$ in \eqref{eq:F(n mod q)}.

\item On \(\Z/M\Z\), we use the discrete Fourier transform
\[
  \widehat f(\xi)
  =
  \sum_{x\pmod* M} f(x)e\!\left(\frac{x\xi}{M}\right),
  \qquad \xi\in\Z/M\Z,
\]
where $f:\Z/M\Z\to\C$ is any function. With this convention,
\[
  f(x)
  =
  \frac1M\sum_{\xi\pmod* M}\widehat f(\xi)
  e\!\left(-\frac{x\xi}{M}\right),
\]
and Parseval's identity is
\[
  \frac1M\sum_{\xi\pmod* M}|\widehat f(\xi)|^2
  =
  \sum_{x\pmod* M}|f(x)|^2.
\]

\item We use Vinogradov and Landau notation in the usual way.  Thus $X\ll Y$ and $X=O(Y)$
both mean that \(|X|\le CY\) for some constant \(C>0\).  Subscripts indicate
the allowed dependence of the implicit constant. 

\item We write $X\asymp Y$
to mean \(X\ll Y\) and \(Y\ll X\).  

\item We write \(x\sim H\) to mean \(H\le x<2H\).

\item For every $t\in\R$, we denote by $\lfloor t\rfloor$ the greatest integer not exceeding $t$, and by $\lceil t\rceil$ the least integer no less than $t$.

\item We write \(\Re z\), \(\Im z\), and \(\arg z\) for the real part,
imaginary part, and argument of a complex number \(z\).  Whenever a principal
branch is mentioned, \(\arg z\) is taken in \((-\pi,\pi]\).  In particular,
\(z^{-1/2}\) denotes the inverse square root with respect to the principal
branch.

\end{itemize}

\section{Defining the weight}
\label{sec:defining-weight}

Throughout, \(H>0\) is an auxiliary scale which will be chosen later in terms of
\(N\).
Fix a nonnegative function \(U\in C_c^\infty(\R)\) with $\supp U=[1,2].$
Define
\begin{equation}\label{eq:BH-def}
 \mathcal B_H(\theta)
  =
  \frac12
  \sum_{v,x,y\in\Z}
  U\!\left(\frac{x}{H}\right)
  U\!\left(\frac{y}{H}\right)
  \exp\!\left(-\frac{v^2}{H^2}\right)
  e\!\left(\theta v^2(x^2+y^2)\right).
\end{equation}
Since $e^{-t}<1/t$ for all $t\ge1$, we have
\begin{equation}\label{eq:B_{H}(theta)}
\sum_{v\ge0}\exp\!\left(-\frac{v^2}{H^2}\right)\ll H+H^2\sum_{v>H}\frac{1}{v^2}\ll H,
\end{equation}
so that $|\mathcal B_H(\theta)|\le \mathcal B_H(0)\ll_U H^3$ for all $\theta\in\T$. Set
\[
  \cF_H(\theta)=\frac{\cB_H(\theta)}{\cB_H(0)}.
\]
Then \(\cF_H\) is the Fourier transform of a positive probability measure supported on $\{0\}$ together with positive integers \(d\) representable as sums of two squares:
\[
  d=v^2(x^2+y^2)=(vx)^2+(vy)^2.
\]
The coefficients $c_d$ from Theorem \ref{thm:power-saving-vdc} will be obtained by collecting the terms in $\cF_H(\theta)$, discarding the negligible contribution from large $d$ and $d=0$, and then renormalizing. In this way the goal becomes proving a lower bound for $\Re \cF_H(\theta)$, or equivalently a lower bound for $\Re \cB_H(\theta)$. 

The form of \(\cB_H\) also matches a pattern suggested by linear
programming experiments. In particular, the data
suggested that favorable coefficients \(c_d\) should be relatively large when
\(d\) has many square divisors. If one ignores the smooth factors
\(U(x/H)U(y/H)\) and \(e^{-v^2/H^2}\), then the coefficient of \(e(d\theta)\)
in \(\cB_H(\theta)\) is proportional to the number of representations of the
form
\[
  d=v^2(x^2+y^2).
\]
Thus square divisors of \(d\) naturally create additional contributions to the
coefficient of \(e(d\theta)\),
matching the feature suggested
by the linear programming experiments. 

To see why the extra variable \(v\) is useful in the proof, we preview the
major-arc estimate. If \(\theta\) is on a major arc, say
\(\theta=a/q+\beta\), then the main term has the form
\[
  C_q(a)W(\beta H^4)\cB_H(0).
\]
Here
\[
  C_q(a)=q^{-3}\sum_{v,x,y\pmod* q}
  e\!\left(\frac{av^2(x^2+y^2)}{q}\right).
\]
A key part of the argument is to prove that this complete sum lies in a fixed
sector:
\[
  |\arg C_q(a)|<\pi/4.
\]
Averaging over the third variable \(v\) is critical here. The more direct
complete sum
\[
  q^{-2}\sum_{x,y\pmod* q}
  e\!\left(\frac{a(x^2+y^2)}{q}\right)
\]
does not lie in the desired sector for many values of \(q\), reflecting the
local asymmetry of sums of two squares mentioned in the introduction. For
\(q=p^m\), the residues with \(p^{\lceil m/2\rceil}\mid v\) satisfy
\(q\mid v^2\), so the phase in \(x,y\) is identically zero and these residues
give a positive real contribution.  We make this precise using the
standard evaluations of quadratic Gauss sums. Combining this calculation with the multiplicativity of the complete sum gives
the desired sector bound for general \(q\). Crucially, the oscillatory integral factor $W(t)$
also satisfies
\[
  |\arg W(t)|\leq \pi/4.
\]
Consequently, \(C_q(a)W(t)\) lies in the closed right half-plane. This makes the
major-arc contribution to \(\Re\cB_H(\theta)\) nonnegative up to the error from
Poisson summation.

This fixed-sector argument is the main distinction between our construction and
the arguments of
\cite{Slijepcevic2010Squares,NincevicSlijepcevic2013}.
Slijep\v{c}evi\'c's construction for squares averages over a carefully chosen
finite collection of dilations so that their real major-arc contributions
combine to give a uniform lower bound.
Nin\v{c}evi\'c--Slijep\v{c}evi\'c note that their odd-polynomial argument does
not directly extend when the complete sum has an uncontrolled imaginary part.

The smooth cutoff
\(U(x/H)U(y/H)\) localizes \(x\) and \(y\) so that \(x,y\sim H\).  This keeps
\(x^2+y^2\asymp H^2\), gives the normalization \(\cB_H(0)\asymp H^3\) in
Lemma \ref{lem:normalization}, and supplies the smoothness needed for Poisson
summation in the major arc estimate. We take \(x,y\sim H\) instead of a cutoff which includes
smaller values of \(x\) and \(y\) because the range
\(x,y\sim H\) already controls the normalization and the minor-arc estimate. In particular it is clean and including the lower ranges would not improve the exponent. 

The factor \(e^{-v^2/H^2}\) localizes $v$ at scale $H$, so that in effect the main contribution to $\cB_H(\theta)$ comes from $v,x,y$ with  \(v^2(x^2+y^2)\ll H^4\), which
is why \(H\) is later chosen to be roughly \(N^{1/4}\).  This Gaussian factor is
also what produces the explicit factor \(W(\beta H^4)\) in the
major arc estimate after summing over \(v\).

We first record a more precise estimate for the
size of \(\cB_H(0)\). We defer the proof to an appendix in order to maintain the flow of the argument. 

\begin{lemma}[Normalization]
\label{lem:normalization}
Let \(U\in C_c^\infty(\R)\) with $\supp U=[1,2]$ and 
\begin{equation}\label{CUdef}
  C_U:=\frac{\sqrt\pi}{2}\left(\int_\R U(t)\,dt\right)^2>0.
\end{equation}
Then, for every \(A>0\),
\begin{equation}
  \cB_H(0)=C_UH^3+O_{U,A}(H^{-A}).
  \label{eq:ZH-asymp}
\end{equation}
In particular \(\cB_H(0)\asymp_U H^3\).
\end{lemma}

The main estimate for the weight is the following.  It will be proved in the
remaining sections.

\begin{theorem}
\label{thm:H-fourier-bound}
Let \(U\in C_c^\infty(\R)\) be as in Lemma \ref{lem:normalization}. For every \(\eps>0\), there are constants \(C_{U,\eps}>0\) and
\(H_0(U,\eps)\) such that, for all \(H\ge H_0(U,\eps)\),
\begin{equation}
  \Re \cF_H(\theta)\ge -C_{U,\eps}H^{-1/2+\eps}
  \qquad(\theta\in\R/\Z).
  \label{eq:H-fourier-bound}
\end{equation}
\end{theorem}

We now show that Theorem \ref{thm:H-fourier-bound} implies Theorem
\ref{thm:power-saving-vdc}.

\begin{proof}[Proof of Theorem \ref{thm:power-saving-vdc} assuming Theorem
\ref{thm:H-fourier-bound}]
It is enough to prove the claim for sufficiently large \(N\), since for bounded
\(N\) we may take \(c_1=1\) and \(c_d=0\) for all $d>1$ and enlarge the constant \(C_\eps\).  Choose
\begin{equation}
  H=\left(\frac{N}{(\log N)^{100}}\right)^{1/4}.
  \label{eq:H-from-N}
\end{equation}
For \(d\ge1\), let
\[
  w_H(d)=
  \sum_{\substack{v\ge1,\ x,y\in\Z\\ v^2(x^2+y^2)=d}}
  U\!\left(\frac{x}{H}\right)U\!\left(\frac{y}{H}\right)e^{-v^2/H^2}.
\]
where \(U\in C_c^\infty(\R)\) is as in Lemma \ref{lem:normalization}. Then \(w_H(d)\ge0\), and \(w_H(d)=0\) unless \(d\) is a sum of two squares.

The only difference between the positive \(v\)-weights \(w_H(d)\) and the
symmetric weight \(\cB_H\) is the contribution coming from \(v=0\).  Put
\[
  z_H=\frac12\sum_{x,y\in\Z}U\!\left(\frac{x}{H}\right)U\!\left(\frac{y}{H}\right).
\]
Since \(U\) is supported in \([1,2]\), we have $z_H\ll_U H^2.$
Moreover, because the two signs \(v\) and \(-v\) cancel the factor \(1/2\), we
have
\[
  \cB_H(\theta)
  =
  z_H+\sum_{d\ge1}w_H(d)e(d\theta),
  \qquad
  \cB_H(0)
  =
  z_H+\sum_{d\ge1}w_H(d).
\]
By Lemma \ref{lem:normalization},
\[
  \frac{z_H}{\cB_H(0)}
  \ll_U
  \frac{H^2}{H^3}
  =
  H^{-1}.
\]
Since \(\cB_H\) is the Fourier transform of a positive measure, we also have $|\cB_H(\theta)|\le \cB_H(0)$. Therefore, we have
\begin{equation}\label{eq:comp1}
\left|
  \frac{\sum_{d\ge1}w_H(d)e(d\theta)}
       {\sum_{d\ge1}w_H(d)}
  -
  \cF_H(\theta)
  \right|                                                   
  \quad =
  \left|
  \frac{\cB_H(\theta)-z_H}{\cB_H(0)-z_H}
  -
  \frac{\cB_H(\theta)}{\cB_H(0)}
  \right|                                                   
  \quad =
  \frac{z_H|\cB_H(\theta)-\cB_H(0)|}
       {\cB_H(0)(\cB_H(0)-z_H)}
  \ll_U H^{-1}
\end{equation}
uniformly for all \(\theta\in\T\). Thus removing the contribution from \(v=0\) changes the normalized Fourier
transform by \(O_U(H^{-1})\).

We now truncate to \(d\le N\).  If \(d>N\), then on the support of
\(U(x/H)U(y/H)\) we have \(x^2+y^2\le 8H^2\), and hence
\[
  \frac{v^2}{H^2}
  =
  \frac{d}{H^2(x^2+y^2)}
  \ge
  \frac{d}{8H^4}
  >
  \frac{N}{8H^4}.
\]
Therefore
\[
  e^{-v^2/H^2}\le e^{-N/(16H^4)}e^{-v^2/(2H^2)}.
\]
Hence, for every fixed \(A>0\),
\[\sum_{d>N}w_H(d)\le
  e^{-N/(16H^4)}
  \sum_{\substack{x,y\in\Z\\v\ge1}}
  U\!\left(\frac{x}{H}\right)U\!\left(\frac{y}{H}\right)
  e^{-v^2/(2H^2)}\ll_{A,U} N^{-A}\sum_{v\ge0}e^{-v^2/(2H^2)}\ll N^{-A}H^3,\]
because \(N/H^4=(\log N)^{100}\), and the last sum is \(O(H^3)\) by \eqref{eq:B_{H}(theta)} with $\sqrt{2}H$ in place of $H$.
Since
\[
  \sum_{d\ge1}w_H(d)=\cB_H(0)-z_H\asymp_U H^3,
\]
we have
\[
  \frac{\sum_{d>N}w_H(d)}{\sum_{d\ge1}w_H(d)}
  \ll_{A,U} N^{-A}.
\]
It follows that, uniformly in \(\theta\),
\begin{equation}\label{eq:comp2}
  \left|
  \frac{\sum_{1\le d\le N}w_H(d)e(d\theta)}
       {\sum_{1\le d\le N}w_H(d)}
  -
  \frac{\sum_{d\ge1}w_H(d)e(d\theta)}
       {\sum_{d\ge1}w_H(d)}
  \right|                                                   
  \le
  \frac{2\sum_{d>N}w_H(d)}{\sum_{d\ge1}w_H(d)}
  \ll_{A,U}N^{-A}.
\end{equation}
Thus removing the terms with \(d>N\) changes the normalized Fourier transform by
\(O_{A,U}(N^{-A})\).

Put
\[
  Z_{N,H}=\sum_{1\le d\le N}w_H(d),
\]
and define
\[
  c_d=
  \begin{cases}
    Z_{N,H}^{-1}w_H(d), & \text{if $1\le d\le N$},\\
    0, & \text{if $d>N$}.
  \end{cases}
\]
Then \(c_d\ge0\) are supported on \(S_N\) and satisfy $\sum_{d\in S_N}c_d=1.$
Furthermore, the two comparison estimates \eqref{eq:comp1} and \eqref{eq:comp2} give, uniformly in \(\theta\),
\[
  \sum_{d\in S_N}c_d e(d\theta)
  =
  \cF_H(\theta)+O_U(H^{-1})+O_{A,U}(N^{-A}).
\]
By Theorem \ref{thm:H-fourier-bound}, applied with \(\eps/2\) in place of
\(\eps\),
\[
  \Re\sum_{d\in S_N}c_d e(d\theta)
  \ge
  -C_{U,\eps}H^{-1/2+\eps/2}
  -O_U(H^{-1})
  -O_{A,U}(N^{-A}).
\]
Since
\[
  H^{-1/2+\eps/2}
  =
  N^{-1/8+\eps/8}
  (\log N)^{100(1/8-\eps/8)}
  \ll_\eps N^{-1/8+\eps},
\]
and also
\[
  H^{-1}
  =
  N^{-1/4}(\log N)^{25}
  \ll_\eps N^{-1/8+\eps},
\]
while \(A\) may be chosen arbitrarily large, we obtain
\[
  \Re\sum_{d\in S_N}c_d e(d\theta)
  \ge
  -C_\eps N^{-1/8+\eps}
  \qquad(\theta\in\R/\Z)
\]
for some suitable constant $C_{\eps}>C_{U,\eps}$. This proves Theorem \ref{thm:power-saving-vdc}.
\end{proof}

\section{The complete sum}

For $(a,q)=1$ define
\begin{equation}
  C_q(a)=q^{-3}\sum_{v,x,y\pmod* q}
  e\!\left(\frac{av^2(x^2+y^2)}{q}\right).
  \label{eq:Cq-def}
\end{equation}
This turns out to be the complete sum which appears in the major arc approximation for $\cB_H(\theta)$. The goal of this section is to show that \(C_q(a)\) always lies in a fixed sector contained in the right half-plane. 

To begin we introduce the notation for the standard quadratic Gauss sum: 
\[G_{q}(a)=\sum_{x\pmod* q}e(ax^2/q).\]
The following computations use only the standard evaluation of quadratic Gauss sums in the case that $q$ is a prime power. We record these standard facts here for reference. For background on these see Chapter $1$ of Berndt--Evans--Williams \cite{BEW}, especially Theorems 1.5.1 and 1.5.2, and Proposition 1.5.3. Given a prime $p$, let $t$ be the largest integer for which $p^t\mid a$ and write $a=p^ta_0$, where $p\nmid a_0$.
\begin{itemize}

\item We have $G_{p^s}(a)=p^tG_{p^{s-t}}(a_0)$ if $s>t$,
while \(G_{p^s}(a)=p^s\) if \(0\le s\le t\).

\item For \(p\) odd and \(p\nmid a_0\),
\[
  G_{p^m}(a_0)=
  \begin{cases}
    p^{m/2}, & \text{if $m$ even},\\[4pt]
    \left(\dfrac{a_0}{p}\right)\varepsilon_p\,p^{m/2},
      & \text{if $m$ odd},
  \end{cases}
\]
where \(\left(\dfrac{a_0}{p}\right)\) is the Legendre symbol modulo $p$ and
\[
  \varepsilon_p=
  \begin{cases}
    1, & \text{if $p\equiv1\pmod4$},\\
    i, & \text{if $p\equiv3\pmod4$}.
  \end{cases}
\]

\item For \(p=2\) and \(a_0\) odd,
\[
  G_{2^m}(a_0)=
  \begin{cases}
    0, & \text{if $m=1$},\\[4pt]
    2^{m/2}\bigl(1+i^{a_0}\bigr), & \text{if $m\ge2$ is even},\\[4pt]
    2^{(m+1)/2}e(a_0/8), & \text{if $m\ge3$ is odd}.
  \end{cases}
\]
\end{itemize}

\begin{proposition}[Odd prime powers]
\label{prop:odd-prime-powers}
Let $a\in\Z$ and $s\in\mathbb{N}$, and let $p$ be an odd prime with $(a,p)=1$. Put $r=\lceil s/2\rceil$ and
\[
  \sigma_{p,s}=\begin{cases}
  1, & \text{if $s$ is even},\\
  \left(\dfrac{-1}{p}\right), & \text{if $s$ is odd}.
  \end{cases}
\]
Then
\begin{equation}
  C_{p^s}(a)=p^{-r}+\sigma_{p,s}(1-p^{-1})\sum_{j=0}^{r-1}p^{-s+j}.
  \label{eq:odd-prime-formula}
\end{equation}
In particular $C_{p^s}(a)$ is real and positive.
\end{proposition}

\begin{proof}
For fixed \(v\pmod*{p^s}\), the sum over \(x,y\) factors as
\[
  \sum_{x,y\pmod*{p^s}}
  e\!\left(\frac{av^2(x^2+y^2)}{p^s}\right)
  =
  G_{p^s}(av^2)^2.
\]
We split the residues \(v\pmod{p^s}\) according to their exact divisibility by
\(p\).  Write \(p^j\parallel v\) to mean that
\(p^j\mid v\) but \(p^{j+1}\nmid v\).

First suppose \(p^j\parallel v\) with \(2j<s\).  Write $v=p^j u$ with $p\nmid u$.
Then $av^2=p^{2j}au^2$ with \(p\nmid au^2\). The standard formulae for quadratic Gauss sums above give
\[
  G_{p^s}(av^2)
  =
  p^{2j}G_{p^{s-2j}}(au^2)
\]
and 
\[
  G_{p^{s-2j}}(au^2)^2
  =
  \begin{cases}
    p^{s-2j}, & \text{if $s-2j$ is even},\\[4pt]
    \left(\dfrac{-1}{p}\right)p^{s-2j}, & \text{if $s-2j$ is odd}.
  \end{cases}
\]
The parity of \(s-2j\) is the same as the parity of \(s\).  Hence
\[
  G_{p^s}(av^2)^2
  =
  p^{4j}G_{p^{s-2j}}(au^2)^2
  =
  \sigma_{p,s}p^{s+2j}.
\]
The number of residues \(v\pmod* {p^s}\) with \(p^j\parallel v\) is
\[
  p^{s-j}-p^{s-j-1}=p^{s-j}(1-p^{-1}).
\]

Now recall that \(r=\lceil s/2\rceil\).  The remaining residues are exactly those
with \(p^r\mid v\).  For these residues we have \(p^s\mid v^2\), so the phase
is identically zero in \(x,y\), and the \(x,y\)-sum is \(p^{2s}\).  There are
\(p^{s-r}\) such residues \(v\pmod* {p^s}\).  Therefore
\[
\begin{aligned}
  C_{p^s}(a)
  &=
  p^{-3s}
  \left(
    p^{s-r}p^{2s}
    +
    \sum_{j=0}^{r-1}
    p^{s-j}(1-p^{-1})\sigma_{p,s}p^{s+2j}
  \right)  \\
  &=
  p^{-r}
  +
  \sigma_{p,s}(1-p^{-1})\sum_{j=0}^{r-1}p^{-s+j},
\end{aligned}
\]
which is \eqref{eq:odd-prime-formula}.

It remains to check positivity.  If \(s\) is even, then
\(\sigma_{p,s}=1\), so every term is positive.  If \(s\) is odd and
\(p\equiv1\pmod4\), then again \(\sigma_{p,s}=1\).  The only case requiring
attention is \(s\) odd and \(p\equiv3\pmod4\), when \(\sigma_{p,s}=-1\).
Write \(s=2r-1\).  Then
\[
  (1-p^{-1})\sum_{j=0}^{r-1}p^{-s+j}
  =
  (1-p^{-1})p^{-s}\frac{p^r-1}{p-1}
  =
  p^{-s-1}(p^r-1).
\]
Since \(s=2r-1\), this equals
\[
  p^{-2r}(p^r-1)=p^{-r}-p^{-2r}.
\]
Therefore in this case
\[
  C_{p^s}(a)
  =
  p^{-r}-(p^{-r}-p^{-2r})
  =
  p^{-2r}>0.
\]
This proves both the formula and the claimed positivity.
\end{proof}

\begin{proposition}[Powers of $2$]
\label{prop:dyadic-sector}
Let $s\in\mathbb{N}$ and let $a\in\Z$ be odd. Put $r=\lceil s/2\rceil$. Then
\[
  \Re C_{2^s}(a)=2^{-r},
  \qquad
  |\Im C_{2^s}(a)|\le 2^{-r}-2^{-s}.
\]
Consequently,
\[
  |\arg C_{2^s}(a)|<\frac\pi4.
\]
\end{proposition}

\begin{proof}
For fixed \(v\pmod* {2^s}\), the sum over \(x,y\) factors as
\[
  \sum_{x,y\pmod* 2^s}
  e\!\left(\frac{av^2(x^2+y^2)}{2^s}\right)
  =
  G_{2^s}(av^2)^2.
\]
We split the residues \(v\pmod* {2^s}\) according to their exact divisibility by
\(2\).  Write \(2^j\parallel v\) to mean that \(2^j\mid v\) but
\(2^{j+1}\nmid v\).

First suppose \(2^j\parallel v\) with \(2j<s\).  Write $v=2^j u$ with $u$ odd.
Then $av^2=2^{2j}au^2$ with $2\nmid au^2$. By the standard formulae for quadratic Gauss sums above, we have
\[
  G_{2^s}(av^2)
  =
  2^{2j}G_{2^{s-2j}}(au^2).
\]
Put \(m=s-2j\).  If \(m=1\), then $G_2(au^2)=0,$
since \(au^2\) is odd. 

If \(m\ge2\), then the standard Gauss-sum evaluation shows that
\(G_{2^m}(au^2)^2\) is purely imaginary.  Indeed, if \(m\) is even, then
\[
  G_{2^m}(au^2)=2^{m/2}\bigl(1+i^{au^2}\bigr),
\]
and since \(au^2\) is odd, we have $\bigl(1+i^{au^2}\bigr)^2\in i\mathbb R.$
If \(m\ge3\) is odd, then
\[
  G_{2^m}(au^2)=2^{(m+1)/2}e(au^2/8),
\]
so
\[
  G_{2^m}(au^2)^2
  =
  2^{m+1}e(au^2/4)\in i\mathbb R.
\]

Now consider the remaining residues, namely those with $2^r\mid v.$
For these residues we have \(2^s\mid v^2\), so the phase is identically zero
in \(x,y\).  Hence the \(x,y\)-sum is \(2^{2s}\).  There are \(2^{s-r}\) such
residues \(v\pmod* {2^s}\), so their normalized contribution to \(C_{2^s}(a)\) is
\[
  2^{-3s}\cdot 2^{s-r}\cdot 2^{2s}
  =
  2^{-r}.
\]
This contribution is real and positive. Since the values of $v$ satisfying $2^r \nmid v$ are purely imaginary, we have 
\[
  \Re C_{2^s}(a)=2^{-r}=2^{-\lceil s/2\rceil}.
\]

It remains to bound the imaginary part of $C_{2^s}(a)$, which arises from the terms $G_{2^s}(av^2)^2=2^{4j}G_{2^{m}}(au^2)^2$ when \(m\ge2\).
The above evaluations give $|G_{2^m}(au^2)^2|=2^{m+1}.$
Therefore
\[
  |G_{2^s}(av^2)^2|
  =
  2^{4j}|G_{2^m}(au^2)^2|
  =
  2^{4j}2^{m+1}
  =
  2^{s+2j+1}.
\]
The number of residues \(v\pmod* {2^s}\) with \(2^j\parallel v\) is $2^{s-j}-2^{s-j-1}=2^{s-j-1}.$
Thus the total normalized magnitude of the contribution from \(2^j\parallel v\) to $C_{2^s}(a)$, in the
case \(m\ge2\), is at most
\[
  2^{-3s}\cdot 2^{s-j-1}\cdot 2^{s+2j+1}
  =
  2^{-s+j}.
\]
Hence,
\[
  |\Im C_{2^s}(a)|
  \le
  \sum_{\substack{0\le j<r\\ s-2j\ge2}}2^{-s+j}=\sum_{j=0}^{\lfloor s/2\rfloor-1}2^{-s+j}=2^{-s+\lfloor s/2\rfloor}-2^{-s}=2^{-r}-2^{-s}<
  2^{-r}=\Re C_{2^s}(a),
\]
since $r+\lfloor s/2\rfloor=\lceil s/2\rceil+\lfloor s/2\rfloor=s$.
As a consequence, $|\arg C_{2^s}(a)|<\frac\pi4.$
\end{proof}

\begin{lemma}[Multiplicativity]
\label{lem:multiplicativity}
If $q=q_1q_2$ with $(q_1,q_2)=1$, and if $(a,q)=1$, then
\[
  C_q(a)=C_{q_1}(a\overline{q_2})C_{q_2}(a\overline{q_1}),
\]
where \(\overline{q_1}\) is the inverse of \(q_1\) modulo \(q_2\) and \(\overline{q_2}\) is the inverse of \(q_2\) modulo \(q_1\), namely, $q_2\overline{q_2}\equiv1\pmod {q_1}$ and
$q_1\overline{q_1}\equiv1\pmod {q_2}$.
\end{lemma}

\begin{proof}
By the Chinese remainder theorem, each residue class modulo \(q=q_1q_2\)
corresponds uniquely to a pair of residue classes modulo \(q_1\) and \(q_2\).
Write
\[
  v\leftrightarrow (v_1,v_2),\qquad
  x\leftrightarrow (x_1,x_2),\qquad
  y\leftrightarrow (y_1,y_2),
\]
where the first component is taken modulo \(q_1\) and the second modulo
\(q_2\). Since
\[
  \frac{\overline{q_2}}{q_1}+\frac{\overline{q_1}}{q_2}
  =
  \frac{q_2\overline{q_2}+q_1\overline{q_1}}{q_1q_2}\equiv \frac{1}{q_1q_2}\pmod*{1},
\]
which follows from \(q_2\overline{q_2}+q_1\overline{q_1}\equiv1\pmod{q_1q_2}\), we have, for any integer \(n\),
\[
  e\!\left(\frac{n}{q_1q_2}\right)
  =
  e\!\left(\frac{\overline{q_2}n}{q_1}\right)
  e\!\left(\frac{\overline{q_1}n}{q_2}\right).
\]
Let us now apply this with \(n=a v^2(x^2+y^2)\).  Modulo \(q_1\), this integer is
congruent to \(a v_1^2(x_1^2+y_1^2)\), and modulo \(q_2\), it is congruent to
\(a v_2^2(x_2^2+y_2^2)\).
Therefore
\[
  e\!\left(\frac{a v^2(x^2+y^2)}{q_1q_2}\right)
  =
  e\!\left(
    \frac{a\overline{q_2}v_1^2(x_1^2+y_1^2)}{q_1}
  \right)
  e\!\left(
    \frac{a\overline{q_1}v_2^2(x_2^2+y_2^2)}{q_2}
  \right).
\]
Hence,
\[
\begin{aligned}
  C_q(a)
  &=
  (q_1q_2)^{-3}
  \sum_{v,x,y\pmod* {q_1q_2}}
  e\!\left(\frac{a v^2(x^2+y^2)}{q_1q_2}\right) \\
  &=
  q_1^{-3}q_2^{-3}
  \sum_{v_1,x_1,y_1\pmod* {q_1}}~
  \sum_{v_2,x_2,y_2\pmod* {q_2}}
  e\!\left(
    \frac{a\overline{q_2}v_1^2(x_1^2+y_1^2)}{q_1}
  \right)
  e\!\left(
    \frac{a\overline{q_1}v_2^2(x_2^2+y_2^2)}{q_2}
  \right) \\
  &=
  \left[
  q_1^{-3}
  \sum_{v_1,x_1,y_1\pmod* {q_1}}
  e\!\left(
    \frac{a\overline{q_2}v_1^2(x_1^2+y_1^2)}{q_1}
  \right)
  \right]
  \left[
  q_2^{-3}
  \sum_{v_2,x_2,y_2\pmod* {q_2}}
  e\!\left(
    \frac{a\overline{q_1}v_2^2(x_2^2+y_2^2)}{q_2}
  \right)
  \right] \\
  &=
  C_{q_1}(a\overline{q_2})C_{q_2}(a\overline{q_1}).
\end{aligned}
\]
This proves the lemma.
\end{proof}

\begin{corollary}
\label{cor:local-sector}
For every $q\ge1$ and $(a,q)=1$,
\[
  |\arg C_q(a)|<\frac\pi4.
\]
In particular $\Re C_q(a)>0$.
\end{corollary}

\begin{proof}
The case \(q=1\) is immediate, since \(C_1(a)=1\).  Now suppose \(q>1\), and
write $q=2^s q_0$ with $q_0 \text{ odd}$, where \(s\ge0\).  By repeated applications of Lemma \ref{lem:multiplicativity}, the
complete sum \(C_q(a)\) factors as a product of complete sums modulo the prime
power divisors of \(q\).  More precisely, the arguments appearing in the
factors are obtained from \(a\) by multiplication by units modulo the
corresponding prime powers, so they remain coprime to those prime powers.

For each odd prime power \(p^k\mid q_0\), Proposition
\ref{prop:odd-prime-powers} shows that the corresponding factor
\(C_{p^k}(\cdot)\) is positive real.  Hence the product of all odd
prime-power factors is positive real.  If \(s=0\), there is no dyadic factor,
and therefore \(C_q(a)\) itself is positive real, and we are done.

If \(s\ge1\), the only possible contribution to the argument comes from the
dyadic factor \(C_{2^s}(\cdot)\).  Since the argument of this dyadic factor is
strictly between \(-\pi/4\) and \(\pi/4\) by Proposition
\ref{prop:dyadic-sector}, and multiplication by a positive real number does
not change the argument, we obtain $|\arg C_q(a)|<\frac{\pi}{4}.$
In particular \(C_q(a)\) lies in the open right half-plane, so $\Re C_q(a)>0.$
\end{proof}

\section{Major arcs}
\label{sec:major}

We define the major arcs as follows: 
\begin{equation}
  \cM=
  \bigcup_{1\le q\le Q}
  \bigcup_{\substack{a\pmod* q\\(a,q)=1}}
  \left\{\theta\in\T:
  \left\|\theta-\frac aq\right\|\le \frac{Q}{qH^4}
  \right\},
  \qquad Q=H^{1-4\sigma},
  \label{eq:maj-arcs}
\end{equation}
where $H\ge1$, and $0<\sigma<1/4$ is fixed. In addition, the main term in the major arc estimate will feature the following function: for $t\in\R$ define
\begin{equation}\label{Wdef}
  W(t)=
  \frac{\displaystyle
  \int_{\R^2}U(X_1)U(X_2)(1-2\pi i t(X_1^2+X_2^2))^{-1/2}\,dX_1\,dX_2}{\displaystyle
  \left(\int_\R U(s)\,ds\right)^2},
\end{equation}
with the inverse square root taking its principal branch.  Since $1-2\pi it(X_1^2+X_2^2)$ lies in the closed right
half-plane, each factor $(1-2\pi it(X_1^2+X_2^2))^{-1/2}$ has argument between
$-\pi/4$ and $\pi/4$.  Since $U\ge0$ and the sector is convex, the same is true of the weighted average
$W(t)$ so that $|\arg W(t)|\le \pi/4.$
In addition we have $|W(t)|\le1$.  Since Corollary \ref{cor:local-sector} gives \(|\arg C_q(a)|<\pi/4\), the
product \(C_q(a)W(t)\) lies in the closed right half-plane.  Thus
\begin{equation}
  \Re\bigl(C_q(a)W(t)\bigr)\ge0
  \qquad((a,q)=1,\,t\in\R).
  \label{eq:local-arch-positive}
\end{equation}

We prove the major arc estimate via Poisson summation. 
\begin{lemma}[Major arc estimate]
\label{lem:lattice-summation}
Let \(J\ge1\) be fixed.  Let \(1\le q\le Q\), \((a,q)=1\), and
\[
  \theta=\frac aq+\beta,
  \qquad
  |\beta|\le \frac{Q}{qH^4}.
\]
Then
\begin{equation}
  \cB_H(\theta)
  =
  C_q(a)W(\beta H^4)\cB_H(0)
  +
  O_{U,J}\!\left(
    \cB_H(0)
    \bigl(q(1+|\beta|H^4)H^{-1}\bigr)^J
  \right).
  \label{eq:major-asymptotic}
\end{equation}
\end{lemma}

\begin{proof}
We start from the definition \eqref{eq:BH-def}:
\[
  \cB_H\!\left(\frac aq+\beta\right)
  =
  \frac12
  \sum_{v,x,y\in\Z}
  e\!\left(\frac aqv^2(x^2+y^2)\right)
  \Phi_\beta(v,x,y),
\]
where
\[
  \Phi_\beta(v,x,y)
  =
  U\!\left(\frac{x}{H}\right)U\!\left(\frac{y}{H}\right)
  \exp\!\left(-\frac{v^2}{H^2}+2\pi i\beta v^2(x^2+y^2)\right)\in\Sch(\R^3).
\]

We split \(v,x,y\) into residue classes modulo \(q\).  On the residue class
\[
  v\equiv r,\qquad x\equiv s,\qquad y\equiv t\pmod q,
\]
the rational phase is constant:
\[
  e\!\left(\frac aqv^2(x^2+y^2)\right)
  =
  e\!\left(\frac aq r^2(s^2+t^2)\right).
\]
Applying Poisson summation to \(\Phi_\beta\) on each residue class (see \eqref{eq:F(n mod q)}) gives
\begin{equation}\label{poisson}
\begin{aligned}
  \cB_H(a/q+\beta)
  &= \frac{1}{2}\sum_{v,x,y\in\Z}
  e\!\left(\frac aqv^2(x^2+y^2)\right)\Phi_\beta(v,x,y) \\
  &=
  \frac{1}{2}q^{-3}\sum_{\ell\in\Z^3}
  \left(
    \sum_{r,s,t\pmod* q}
    e\!\left(
      \frac{ar^2(s^2+t^2)+\ell_1r+\ell_2s+\ell_3t}{q}
    \right)
  \right)
  \widehat{\Phi_\beta}\!\left(\frac{\ell}{q}\right).
\end{aligned}
\end{equation}
We claim the zero frequency \(\ell=0\) gives the main term.  Notice that
\[
  q^{-3}
  \sum_{r,s,t\pmod* q}
  e\!\left(\frac{ar^2(s^2+t^2)}q\right)
  =
  C_q(a)
\]
by definition, so the contribution from $\ell=0$ in \eqref{poisson} is $\frac12 C_q(a)\widehat{\Phi_\beta}(0).$ We now evaluate \(\widehat{\Phi_\beta}(0)\).  Scaling
\[
  x=HX_1,\qquad y=HX_2,\qquad v=HT,
\]
gives
\[
  \frac12\widehat{\Phi_\beta}(0)
  =
  \frac12H^3
  \int_{\R^3}
  U(X_1)U(X_2)
  \exp\!\left(
    -T^2+2\pi i\beta H^4T^2(X_1^2+X_2^2)
  \right)
  \,dT\,dX_1\,dX_2.
\]
For fixed \(X_1,X_2\),
\[
  \int_\R
  \exp\!\left(-(1-2\pi i\beta H^4(X_1^2+X_2^2))T^2\right)\,dT
  =
  \sqrt\pi(1-2\pi i\beta H^4(X_1^2+X_2^2))^{-1/2}.
\]
which is a consequence of the classical formula
\[\int_\R e^{-zu^2}\,du=\sqrt{\frac{\pi}{z}}\]
for all $z\in\C$ with $\Re(z)>0$. Therefore, by the definition of \(C_U\) and \(W\) (\eqref{CUdef} and \eqref{Wdef}), we have
\[
  \frac12\widehat{\Phi_\beta}(0)
  =
  W(\beta H^4)C_UH^3.
\]
By Lemma \ref{lem:normalization} applied with \(A=J\),
\[
  C_UH^3
  =
  \cB_H(0)
  +
  O_{U,J}(H^{-J}).
\]
Since \(|C_q(a)|\le1\), \(|W(\beta H^4)|\le1\), and
\[
  H^{-J}
  \ll_U
  \cB_H(0)
  \bigl(q(1+|\beta|H^4)H^{-1}\bigr)^J,
\]
the contribution to \eqref{poisson} from $\ell=0$ is
\[
  C_q(a)W(\beta H^4)\cB_H(0)
  +
  O_{U,J}\!\left(
    \cB_H(0)
    \bigl(q(1+|\beta|H^4)H^{-1}\bigr)^J
  \right).
\]

It remains to bound the contribution from $\ell\neq 0$. By scaling again,
\[
  x=HX_1,\qquad y=HX_2,\qquad v=HT,
\]
 \(\Phi_\beta(v,x,y)\) becomes
\begin{equation}\label{scaledPhi}
  U(X_1)U(X_2)
  \exp\!\left(-T^2+2\pi i\beta H^4T^2(X_1^2+X_2^2)\right).
\end{equation}
The immediate goal is to bound \(|\widehat{\Phi_\beta}(\xi)|\) using
integration by parts. Fix a large positive integer $M>J$. 
We estimate its derivatives of total order at most \(M\).  Derivatives of \(U(X_1)U(X_2)\) are of size \(O_{U,M}(1)\).  For the exponential factor, each
derivative produces factors coming from derivatives of $-T^2+2\pi i\beta H^4T^2(X_1^2+X_2^2)$ On the support of \(U(X_1)U(X_2)\), the variables \(X_1,X_2\) remain in the fixed compact set \(\supp U\subseteq[1,2]\).  Hence all polynomial factors in
\(X_1,X_2\) are bounded.  The derivatives of the phase are bounded by powers
of \(1+|2\pi i\beta H^4|\) times powers of \(1+|T|\), and after at most \(M\)
derivatives the resulting \(T\)-factor is $O_{U,M}((1+|T|)^{2M})$.  Since
\[
  \left|\exp(-T^2+2\pi i\beta H^4 T^2(X_1^2+X_2^2))\right|=e^{-T^2},
\]
we obtain, for \(j_0+j_1+j_2\le M\),
\[
  \left|
  \partial_T^{j_0}\partial_{X_1}^{j_1}\partial_{X_2}^{j_2}
  \left[
  U(X_1)U(X_2)
  \exp\!\left(-T^2+2\pi i\beta H^4T^2(X_1^2+X_2^2)\right)
  \right]
  \right|
  \ll_{U,M}
  (1+|\beta|H^4)^M(1+|T|)^{2M}e^{-T^2}.
\]
The left-hand side is supported in a fixed compact set in the
\((X_1,X_2)\)-variables, so this bound is integrable in \(T,X_1,X_2\).  Scaling
back to \(v,x,y\), we get
\[
  \left\|
  \partial_v^{j_0}\partial_x^{j_1}\partial_y^{j_2}\Phi_\beta
  \right\|_1
  \ll_{U,M}
  H^{3-j_0-j_1-j_2}(1+|\beta|H^4)^M
  \qquad(j_0+j_1+j_2\le M).
\]

Now suppose \(H|\xi|\ge1\). Choose a coordinate \(\xi_j\) with
\(|\xi_j|\gg|\xi|\). Integrating by parts \(M\) times in the corresponding coordinate \(z_j\) gives
\[
  \widehat{\Phi_\beta}(\xi)
  =
  (2\pi i\xi_j)^{-M}
  \int_{\R^3}
  \partial_{z_j}^M\Phi_\beta(z)e(-z\cdot \xi)\,dz,
\]
where \(z=(v,x,y)\in\R^3\). Hence
\[
  |\widehat{\Phi_\beta}(\xi)|
  \ll_{U,M}
  |\xi_j|^{-M}H^{3-M}(1+|\beta|H^4)^M
  \ll_{U,M}
  H^3(1+|\beta|H^4)^M(H|\xi|)^{-M}.
\]
Combining this with the trivial bound 
\[|\widehat{\Phi_\beta}(\xi)|\le \|\Phi_\beta\|_1\ll_U H^3,\qquad \forall\xi\in\R^3,\]
gives
\[
  |\widehat{\Phi_\beta}(\xi)|
  \ll_{U,M}
  H^3(1+|\beta|H^4)^M(1+H|\xi|)^{-M},\qquad \forall\xi\in\R^3.
\]
Thus the total contribution from $\ell\neq 0$ to \eqref{poisson} is 
\[
\begin{aligned}
  &\ll_{U,M}
  H^3(1+|\beta|H^4)^M
  \sum_{\ell\ne0}
  \left(1+\frac{H|\ell|}{q}\right)^{-M},
\end{aligned}
\]
where we used the trivial bound $O(q^3)$ for the complete exponential sum over $r,s,t\pmod*{q}$. Since
\[
  \sum_{\ell\ne0}
  \left(1+\frac{H|\ell|}{q}\right)^{-M}
  \le
  \left(\frac qH\right)^M
  \sum_{\ell\ne0}|\ell|^{-M}
  \ll_M
  \left(\frac qH\right)^M,
\]
the total contribution from \(\ell\neq0\) in \eqref{poisson} is
\[
  \ll_{U,M}
  H^3
  \left(\frac{q(1+|\beta|H^4)}{H}\right)^M.
\]
Since \(q\le Q\) and \(|\beta|H^4\le Q/q\), we have
\[
  q(1+|\beta|H^4)\le q+Q\ll Q=H^{1-4\sigma}.
\]
Therefore, for sufficiently large \(M>J\), this is
\[
  O_{U,J}\!\left(
    H^3\bigl(q(1+|\beta|H^4)H^{-1}\bigr)^J
  \right).
\]

Combining the zero-frequency contribution from $\ell=0$ and the contributions from $\ell\neq 0$ in \eqref{poisson} gives
\[
  \cB_H\!\left(\frac aq+\beta\right)
  =
  C_q(a)W(\beta H^4)\cB_H(0)
  +
  O_{U,J}\!\left(
    \cB_H(0)
    \bigl(q(1+|\beta|H^4)H^{-1}\bigr)^J
  \right).
\]
This proves \eqref{eq:major-asymptotic}.
\end{proof}

\begin{proposition}[Major-arc lower bound]
\label{prop:major-lower}
Fix $0<\sigma<1/4$ and take $Q=H^{1-4\sigma}$. Uniformly for $\theta\in\cM$,
\begin{equation}
  \Re \cF_H(\theta)\ge -C_{U,\sigma}H^{-1/2}.
  \label{eq:major-lower}
\end{equation}
\end{proposition}

\begin{proof}
Let $J=1/(8\sigma)$. Dividing the estimate in Lemma \ref{lem:lattice-summation} by \(\cB_H(0)\),
the normalized error is
\[
  \ll_{U,J} \bigl(q(1+|\beta|H^4)H^{-1}\bigr)^J.
\]
On \(\cM\) we have \(q\le Q\) and \(|\beta|H^4\le Q/q\), hence
\[
  q(1+|\beta|H^4)\ll q+Q\ll Q.
\]
Thus the normalized error is
\[
  \ll_{U,J}(QH^{-1})^J=H^{-4J\sigma}=H^{-1/2},
\]
by our choice of \(J\). The real part of
the normalized main term is nonnegative by \eqref{eq:local-arch-positive}.
\end{proof}

\section{Minor arcs}
\label{sec:minor}
We record the following immediate consequence of Dirichlet's approximation principle. 
\begin{lemma}
\label{lem:minor-dirichlet}
If $\theta\notin\cM$, then there are coprime integers $a,q$ such that
\begin{equation}
  Q<q\le \frac{H^4}{Q},
  \qquad
  \left|\theta-\frac aq\right|\le \frac{Q}{qH^4}\le \frac1{q^2}.
  \label{eq:minor-dirichlet}
\end{equation}
\end{lemma}

We now prove a Weyl-type estimate. 
\begin{lemma}
\label{lem:averaged-sums}
Let $H\ge2$ and $1\le M\le 80H^2\log H$.  Suppose that
$\theta$ and $q$ satisfy \eqref{eq:minor-dirichlet}.  Then for every $\eps>0$,
\begin{equation}
  \sum_{1\le |n|\le M}
  \left|
  \sum_{x,y\in\Z}
  U\!\left(\frac{x}{H}\right)U\!\left(\frac{y}{H}\right)
  e\!\left(n\theta(x^2+y^2)\right)
  \right|
  \ll_{U,\eps}
  H^\eps
  \left(
    MH+H^2+q+\frac{H^2M}{q}
  \right).
  \label{eq:averaged-sums}
\end{equation}
\end{lemma}

\begin{proof}
For any $\alpha\in\R$, put
\[
  S(\alpha)=\sum_{x\in\Z}U\!\left(\frac{x}{H}\right)e(\alpha x^2).
\]
Thus the left-hand side of \eqref{eq:averaged-sums} is simply
\[
  \sum_{1\le |n|\le M}|S(n\theta)|^2.
\]
Since \(U\) is real-valued, we have \(S(-\alpha)=\overline{S(\alpha)}\), and hence
\[
  \sum_{1\le |n|\le M}|S(n\theta)|^2
  =
  2\sum_{1\le n\le M}|S(n\theta)|^2.
\]
It is therefore enough to estimate the sum over positive \(n\).

Expanding the square gives
\[
\begin{aligned}
  \sum_{1\le n\le M}|S(n\theta)|^2
  &=
  \sum_{1\le n\le M}
  \sum_{x,y\in\Z}
  U\!\left(\frac{x}{H}\right)
  U\!\left(\frac{y}{H}\right)
  e\!\left(n\theta(x^2-y^2)\right)  \\
  &=
  \sum_{x,y\in\Z}
  U\!\left(\frac{x}{H}\right)
  U\!\left(\frac{y}{H}\right)
  \sum_{1\le n\le M}e\!\left(n\theta(x^2-y^2)\right).
\end{aligned}
\]
Here the sums are finite because \(U\) is compactly supported.  Since
\(\supp U\subseteq [1,2]\), the variables \(x\) and \(y\) which occur are both
positive and satisfy \(x,y\asymp H\).  Therefore \(x^2=y^2\) is the same as
\(x=y\).  The diagonal contribution is
\[
  M\sum_{x\in\Z}U\!\left(\frac{x}{H}\right)^2
  \ll_U MH.
\]

We now consider the off-diagonal terms.  Put $m=x^2-y^2\ne0$.
Since \(x,y\asymp H\), we have \(1\le |m|\ll H^2\).
For any fixed \(m\ne0\), the number of pairs \((x,y)\) which can contribute
is \(O_\eps(H^{\eps/2})\).  Indeed,
\[
  m=x^2-y^2=(x-y)(x+y),
\]
and once the two factors \(x-y\) and \(x+y\) are chosen, the values of \(x\)
and \(y\) are determined.  Thus the number of such pairs is bounded by a
divisor function of \(|m|\), and the divisor bound gives
\[
  \tau(|m|)\ll_\eps |m|^{\eps/4}\ll H^{\eps/2},
\]
after adjusting the implicit constant.

For each fixed \(m\), the inner sum over \(n\) is a finite geometric progression, so
\[
  \left|\sum_{1\le n\le M}e(nm\theta)\right|
  \ll
  \min\left\{M,\|m\theta\|^{-1}\right\},
\]
with the usual convention that the right-hand side is \(M\) when
\(\|m\theta\|=0\).  Hence
\begin{equation}
  \sum_{1\le n\le M}|S(n\theta)|^2
  \ll_{U,\eps}
  MH
  +
  H^{\eps/2}
  \sum_{1\le |m|\ll H^2}
  \min\left\{M,\|m\theta\|^{-1}\right\}.
  \label{eq:l2-reduction}
\end{equation}
From now on we may assume \(M\ge3\), since if \(M<3\), then the right-hand side above is $\ll MH+H^{2+\eps/2}$, which is acceptable.

We next estimate the remaining sum over \(m\). According to \eqref{eq:minor-dirichlet}, we may write
\[
  \theta=\frac aq+\rho,
  \qquad
  |\rho|\le q^{-2},
  \qquad
  (a,q)=1.
\]
We claim that for \(X\ge1\) and \(0<\delta\le1/2\),
\begin{equation}
  \#\{1\le |m|\le X:\|m\theta\|\le\delta\}
  \ll\delta X+\frac{X}{q}+q\delta+1.
  \label{eq:spacing-count}
\end{equation}
To see this, it suffices to count positive \(m\).  In any interval \(I\) of \(q\)
consecutive integers, the values \(ma/q\) run through a translate of the
\(q\) equally spaced points \(j/q\).  If \(I=\{L+1,\dots,L+q\}\), then
\[
  m\rho=L\rho+O(q|\rho|)=L\rho+O(1/q)
  \qquad(m\in I).
\]
Therefore the condition \(\|m\theta\|\le\delta\) implies
\[
  \left\|\frac{ma}{q}+L\rho\right\|\le \delta+\frac1q.
\]
So the points $ma/q~(m\in I)$ fall inside an arc of $\T$ of length at most $\delta+1/q$. The number of such $m\in I$ does not exceed $q(\delta+1/q)+1=q\delta+1$. Since $[1,X]$ is the disjoint union of at most $X/q+1$ such intervals, we have
\[\#\{1\le m\le X:\|m\theta\|\le\delta\}\le(q\delta+1)(X/q+1)=\delta X+\frac{X}{q}+q\delta+1,\]
which immediately implies \eqref{eq:spacing-count}.

Now we apply \eqref{eq:spacing-count} with \(X\ll H^2\). We use the elementary
dyadic decomposition
\[
  \min\left\{M,\|m\theta\|^{-1}\right\}
  \ll
  1+
  \sum_{\substack{2\le R\le M\\ R\ {\rm dyadic}}}
  R\,1_{\{\|m\theta\|\le R^{-1}\}}.
\]
Indeed, if $\|m\theta\|\le1/M$, then both sides are $\asymp M$. If $\|m\theta\|>1/M$, then $1/(R_0+1)<\|m\theta\|\le 1/R_0$ for some dyadic endpoint $2\le R_0<M$, so that both sides are $\asymp R_0$. Summing this inequality over \(1\le |m|\ll H^2\)
gives
\[
\begin{aligned}
  \sum_{1\le |m|\ll H^2}
  \min\left\{M,\|m\theta\|^{-1}\right\}
  &\ll
  H^2
  +
  \sum_{\substack{2\le R\le M\\ R\ {\rm dyadic}}}
  R~
  \#\{1\le |m|\ll H^2:\|m\theta\|\le R^{-1}\}.
\end{aligned}
\]
Using \eqref{eq:spacing-count} with \(\delta=R^{-1}\), the summand is
\[
\begin{aligned}
  R~
  \#\{1\le |m|\ll H^2:\|m\theta\|\le R^{-1}\}
  &\ll
  R\left(\frac{H^2}{R}+\frac{H^2}{q}+\frac qR+1\right)  \\
  &=
  H^2+\frac{H^2R}{q}+q+R.
\end{aligned}
\]
Therefore,
\[
\begin{aligned}
  \sum_{1\le |m|\ll H^2}
  \min\left\{M,\|m\theta\|^{-1}\right\}
  &\ll
  H^2
  +
  (H^2+q)\log M
  +
  \frac{H^2M}{q}
  +
  M.
\end{aligned}
\]
Since \(1\le M\le 80H^2\log H\) by our assumption, all but the third term on the right-hand side 
are dominated by \(H^{\eps/2}(H^2+q)\).  Hence
\[
  \sum_{1\le |m|\ll H^2}
  \min\left\{M,\|m\theta\|^{-1}\right\}
  \ll_{\eps}
  H^{\eps/2}
  \left(
    H^2+q+\frac{H^2M}{q}
  \right).
\]
Inserting this into \eqref{eq:l2-reduction} gives
\[
  \sum_{1\le n\le M}|S(n\theta)|^2
  \ll_{U,\eps}
  MH
  +
  H^\eps
  \left(
    H^2+q+\frac{H^2M}{q}\right)\le H^\eps
  \left(MH+
    H^2+q+\frac{H^2M}{q}
  \right),
\]
which is exactly \eqref{eq:averaged-sums}.
\end{proof}

\begin{proposition}[Minor-arc estimate]
\label{prop:minor}
For every \(\eps>0\) and \(0<\sigma<1/4\), we have 
\begin{equation}
  |\cF_H(\theta)|
  \ll_{U,\eps,\sigma}
  H^{-1/2+2\sigma+\eps}
  \label{eq:minor-final}
\end{equation}
uniformly for \(\theta\notin\cM\).
\end{proposition}

\begin{proof}
We may suppose that $H$ is sufficiently large. Since \(\theta\notin\cM\), Lemma \ref{lem:minor-dirichlet} gives coprime
integers \(a,q\) such that
\[
  Q<q\le \frac{H^4}{Q},
  \qquad
  \left|\theta-\frac aq\right|\le \frac{Q}{qH^4}\le \frac1{q^2}.
\]
Put $V=H(40\log H)^{1/2}$ and 
let
\[
  \mathcal S(\theta)
  =
  \frac12
  \sum_{x,y\in\Z}
  U\!\left(\frac{x}{H}\right)U\!\left(\frac{y}{H}\right)
  \sum_{|v|\le V}
  \exp\!\left(-\frac{v^2}{H^2}\right)
  e\!\left(\theta v^2(x^2+y^2)\right).
\]
The discarded range \(|v|>V\) is negligible.  Indeed,
\begin{equation}\label{eq:B_H-S}
|\cB_H(\theta)-\mathcal S(\theta)|\ll_U
  H^2\sum_{|v|>V}e^{-v^2/H^2}<2H^2\int_{V/2}^{\infty}e^{-t^2/H^2}\,dt\ll H^3e^{-V^2/(8H^2)}\ll H^{-2}.   
\end{equation}
Since \(\cB_H(0)\asymp_U H^3\) by Lemma \ref{lem:normalization}, this error is negligible after normalization.

We now estimate \(\mathcal S(\theta)\).  By Cauchy's inequality in the
\((x,y)\)-variables,
\[|\mathcal S(\theta)|^2\le
  \left(
    \sum_{x,y\in\Z}
    U\!\left(\frac{x}{H}\right)U\!\left(\frac{y}{H}\right)
  \right)                                     
  \sum_{x,y\in\Z}
  U\!\left(\frac{x}{H}\right)U\!\left(\frac{y}{H}\right)
  \left|
    \frac12
    \sum_{|v|\le V}
    \exp\!\left(-\frac{v^2}{H^2}\right)
    e\!\left(\theta v^2(x^2+y^2)\right)
  \right|^2 .
\]
The first factor is \(O_U(H^2)\).  Expanding the square in the second factor
and using
\[
  \exp\!\left(-\frac{v_1^2+v_2^2}{H^2}\right)\le1,
\]
we find
\begin{equation}
  |\mathcal S(\theta)|^2
  \ll_U
  H^2
  \sum_{|v_1|,|v_2|\le V}
  \left|
  \sum_{x,y\in\Z}
  U\!\left(\frac{x}{H}\right)U\!\left(\frac{y}{H}\right)
  e\!\left(\theta(v_1^2-v_2^2)(x^2+y^2)\right)
  \right|.
  \label{eq:minor-full-ttstar}
\end{equation}

We first handle the terms with \(v_1^2=v_2^2\).  There are \(O(V)\) such pairs
\((v_1,v_2)\), and for each of them the inner sum over \(x,y\) is \(O_U(H^2)\).
Hence these terms contribute \(O_U(VH^2)\) to the sum on the right-hand side of
\eqref{eq:minor-full-ttstar}.

For the remaining terms, put $n=v_1^2-v_2^2\ne0$.
Then 
\[
  1\le |n|\le V^2=40H^2\log H.
\]
For each fixed nonzero \(n\), the number of pairs \((v_1,v_2)\) with
\(|v_1|,|v_2|\le V\) and \(v_1^2-v_2^2=n\) is \(O_\eps(H^{\eps/2})\) by the
divisor bound. Grouping the off-diagonal terms according to the value of \(n\), their total
contribution is therefore
\[
\begin{aligned}
  &\ll_\eps
  H^{\eps/2}
  \sum_{1\le |n|\le V^2}
  \left|
  \sum_{x,y\in\Z}
  U\!\left(\frac{x}{H}\right)U\!\left(\frac{y}{H}\right)
  e\!\left(n\theta(x^2+y^2)\right)
  \right|.
\end{aligned}
\]
Let \(M=\lceil V^2\rceil\).  For \(H\) sufficiently large, we have
\(M\le 80H^2\log H\), so Lemma \ref{lem:averaged-sums}, applied with
\(\eps/2\) in place of \(\eps\), gives
\[\sum_{1\le |n|\le V^2}
  \left|
  \sum_{x,y\in\Z}
  U\!\left(\frac{x}{H}\right)U\!\left(\frac{y}{H}\right)
  e\!\left(n\theta(x^2+y^2)\right)
  \right|\ll_{U,\eps}
  H^{\eps/2}
  \left(
    MH+H^2+q+\frac{H^2M}{q}
  \right).
\]
Since \(M\asymp V^2\), the diagonal and off-diagonal contributions together give
\[
\begin{aligned}
  \sum_{|v_1|,|v_2|\le V}
  \left|
  \sum_{x,y\in\Z}
  U\!\left(\frac{x}{H}\right)U\!\left(\frac{y}{H}\right)
  e\!\left(\theta(v_1^2-v_2^2)(x^2+y^2)\right)
  \right|
  &\ll_{U,\eps}
  VH^2
  +
  H^\eps
  \left(
    V^2H+H^2+q+\frac{H^2V^2}{q}
  \right)  \\
  &\ll_{U,\eps}
  H^\eps
  \left(
    H^3+q+\frac{H^4}{q}
  \right),
\end{aligned}
\]
where the powers of \(\log H\) coming from \(V\) have been absorbed into
\(H^\eps\).

Inserting this bound into \eqref{eq:minor-full-ttstar} and using
\(\cB_H(0)\asymp_U H^3\), we obtain
\[\left|\frac{\mathcal S(\theta)}{\cB_H(0)}\right|^2
  \ll_{U,\eps}
  \frac{H^2}{H^6}
  H^\eps
  \left(
    H^3+q+\frac{H^4}{q}
  \right)=
  H^\eps
  \left(
    H^{-1}+\frac{q}{H^4}+\frac1q
  \right).
\]
By the choice of \(q\), we have $q>Q$ and $q\le \frac{H^4}{Q}$. It follows that
\[
  \left|\frac{\mathcal S(\theta)}{\cB_H(0)}\right|^2
  \ll_{U,\eps}
  H^\eps\left(H^{-1}+Q^{-1}\right).
\]
Using \(\cB_H(0)\asymp_U H^3\) and appealing to \eqref{eq:B_H-S}, we get
\[
  |\cF_H(\theta)|^2
  \ll_{U,\eps}
  H^\eps\left(H^{-1}+Q^{-1}\right).
\]
Taking square roots and absorbing the resulting change in \(\eps\), we get
\[
  |\cF_H(\theta)|
  \ll_{U,\eps}
  H^\eps\left(H^{-1/2}+Q^{-1/2}\right).
\]
Thus, since \(Q=H^{1-4\sigma}\), we have
\[
  |\cF_H(\theta)|
  \ll_{U,\eps,\sigma}
  H^{-1/2+2\sigma+\eps}.
\]
This proves \eqref{eq:minor-final}.
\end{proof}

\section{Finishing the argument}
We now combine Proposition \ref{prop:major-lower} and Proposition
\ref{prop:minor} to give a quick proof of Theorem \ref{thm:H-fourier-bound}.
\begin{proof}[Proof of Theorem \ref{thm:H-fourier-bound}]
Fix $0<\eps<1$ and take $0<\sigma<\eps/4$.
We work with \(H\) sufficiently large.
If $\theta\in\cM$, then Proposition
\ref{prop:major-lower} gives
\[
  \Re \cF_H(\theta)\ge -C_{U,\sigma}H^{-1/2}.
\]
If $\theta\notin\cM$, then Proposition \ref{prop:minor}, applied with
\(\eps/2\) in place of \(\eps\), gives
\[
  \Re \cF_H(\theta)
  \ge
  -|\cF_H(\theta)|
  \ge
  -C_{U,\eps,\sigma}H^{-1/2+2\sigma+\eps/2}.
\]
Since \(\sigma<\eps/4\), we have
\[
  -\frac12+2\sigma+\frac\eps2
  <
  -\frac12+\eps.
\]
Absorbing the dependence on \(\sigma\) into the constant, we obtain
\[
  \Re \cF_H(\theta)\ge -C_{U,\eps}H^{-1/2+\eps}
  \qquad(\theta\in\R/\Z),
\]
for some constant $C_{U,\eps}>0$. This proves Theorem \ref{thm:H-fourier-bound}.
\end{proof}

\begin{appendix}
\section{Proof of the Sárközy-type theorem}
We record the standard proof that power-saving van der Corput property implies the Sárközy-type result promised in the introduction. For this we use Fourier analysis on the finite cyclic group \(\Z/M\Z\). For a function $f:\Z/M\Z\to \C$, we use the Fourier transform
\[
  \widehat f(\xi)
  =
  \sum_{x\pmod* M} f(x)e\!\left(\frac{x\xi}{M}\right),
  \qquad \xi\in\Z/M\Z.
\]
Recall that Parseval's identity yields
\[
  \frac1M\sum_{\xi\pmod* M}|\widehat f(\xi)|^2
  =
  \sum_{x\pmod* M}|f(x)|^2.
\]

\begin{proof}[Proof of Corollary \ref{cor:sarkozy-two-squares}]
Let \(M=2N+1\), and regard \(A\) as a subset of \(\Z/M\Z\).  We apply Theorem
\ref{thm:power-saving-vdc} to obtain a trigonometric polynomial
\[
  P(\theta)=\sum_{d\in S_N}c_d e(d\theta),
  \qquad
  \eta=C_\eps N^{-1/8+\eps},
\]
such that \(\Re P(\theta)\ge -\eta\) for every \(\theta\in\R/\Z\), and \(P(0)=1\).
Let
\[
  \widehat{1_A}(\xi)=\sum_{a\in A}e(a\xi/M),
  \qquad \xi\in\Z/M\Z.
\]
Since \(A\subseteq [N]\), \(1\le d\le N\), and \(M=2N+1\), the congruence
\(a'-a\equiv d\pmod M\) is the same as the equality \(a'-a=d\).  Hence the
assumption \((A-A)\cap S_N=\varnothing\) gives
\[
  \frac1M\sum_{\xi\pmod* M}|\widehat{1_A}(\xi)|^2P(\xi/M)=0.
\]
Taking real parts and separating the zero frequency, we get
\[
  0
  =
  \frac{|A|^2}{M}
  +
  \frac1M\sum_{\substack{\xi\pmod* M\\ \xi\ne0}}
  |\widehat{1_A}(\xi)|^2\Re P(\xi/M).
\]
Using \(\Re P(\theta)\ge -\eta\) and Parseval,
\[
  0
  \ge
  \frac{|A|^2}{M}
  -
  \frac{\eta}{M}\sum_{\xi\pmod* M}|\widehat{1_A}(\xi)|^2
  =
  \frac{|A|^2}{M}-\eta |A|.
\]
This implies \(|A|\le \eta M\ll_\eps N^{7/8+\eps}\), as desired.
\end{proof}

\section{The normalization factor}
We now prove Lemma \ref{lem:normalization}. This is a straightforward application of Poisson summation. 
\begin{proof}[Proof of Lemma \ref{lem:normalization}]
Define
\[
  \Phi(T,X_1,X_2)
  =
  U(X_1)U(X_2)\exp\!\left(-T^2\right).
\]
Then
\[
  \cB_H(0)
  =
  \frac12
  \sum_{v,x,y\in\Z}
  \Phi\!\left(\frac vH,\frac xH,\frac yH\right).
\]
The function \(\Phi\) is smooth, compactly supported in the \((X_1,X_2)\)
variables, and rapidly decaying in the \(T\)-variable. Therefore $\Phi\in \Sch(\R^3)$, and 
Poisson summation on \(\Z^3\) gives
\[
  \cB_H(0)
  =
  \frac12 H^3
  \sum_{\ell\in\Z^3}\widehat\Phi(H\ell).
\]
The zero frequency contributes
\[
  \frac12 H^3\widehat\Phi(0)
  =
  \frac12 H^3
  \int_{\R^3}
  U(X_1)U(X_2)
  \exp\!\left(-T^2\right)
  \,dT\,dX_1\,dX_2 .
\]
Since $\int_\R e^{-T^2}\,dT=\sqrt\pi,$
the zero frequency is $C_UH^3,$
where
\[
  C_U=
  \frac{\sqrt\pi}{2}\left(\int_\R U(t)\,dt\right)^2.
\]

It remains to bound the nonzero frequencies.  Since $\Phi\in \Sch(\R^3)$ is Schwartz, for
every \(B>0\),
\[
  |\widehat\Phi(\xi)|\ll_{U,B}(1+|\xi|)^{-B}.
\]
Choosing \(B\) sufficiently large in terms of \(A\), we obtain
\[
\begin{aligned}
  H^3\sum_{\ell\ne0}|\widehat\Phi(H\ell)|\ll_{U,B}
  H^3\sum_{\ell\ne0}(1+H|\ell|)^{-B} \ll_{U,A} H^{-A}.
\end{aligned}
\]
Hence
\[
  \cB_H(0)=C_UH^3+O_{U,A}(H^{-A}).
\]
This proves \eqref{eq:ZH-asymp}.
\end{proof}

\end{appendix}
\section*{Acknowledgments}

The second author was supported by the HUN-REN Alfr\'ed R\'enyi Institute of Mathematics
(Erd\H{o}s Center) during the completion of this manuscript, and thanks the
R\'enyi Institute for its hospitality. The second author would also like to thank Alex Rice for suggesting the problem and 
\'Akos Magyar for useful conversations.

\section*{AI disclosure}
The authors used OpenAI's ChatGPT, specifically GPT-5.5 Pro, during the research underlying this paper, primarily as a sounding board, to compute and analyze examples, to investigate proposed intermediate claims, and to test parameter choices. In particular, we used GPT-5.5 Pro to analyze candidate choices for the weights $c_d$. ChatGPT was used for LaTeX assistance, formatting, and copyediting during the preparation of the manuscript. The mathematical arguments presented here are the authors' own, and the authors take full responsibility for the contents of the paper.

\bibliographystyle{alpha}
\bibliography{references}

\begin{thebibliography}{BEW98}

\bibitem[BEW98]{BEW}
B.~C. Berndt, R.~J. Evans, and K.~S. Williams.
\newblock {\em {Gauss} and {Jacobi} Sums}.
\newblock Canadian Mathematical Society Series of Monographs and Advanced Texts. John Wiley \& Sons, Inc., New York, 1998.

\bibitem[Gre24]{GreenShiftedPrimes}
B.~Green.
\newblock On {S}{\'a}rk{\"o}zy's theorem for shifted primes.
\newblock {\em Journal of the American Mathematical Society}, 37(4):1121--1201, 2024.

\bibitem[GS24]{GreenSawhneyFurstenbergSarkozy}
B.~Green and M.~Sawhney.
\newblock New bounds for the {F}urstenberg--{S}{\'a}rk{\"o}zy theorem, 2024.
\newblock arXiv:2411.17448.

\bibitem[KMF78]{KamaeMendesFrance}
T.~Kamae and M.~Mend{\`e}s~France.
\newblock Van der {C}orput's difference theorem.
\newblock {\em Israel Journal of Mathematics}, 31(3--4):335--342, 1978.

\bibitem[Mon94]{MontgomeryTenLectures}
{H.~L.} Montgomery.
\newblock {\em Ten Lectures on the Interface between Analytic Number Theory and Harmonic Analysis}, volume~84 of {\em CBMS Regional Conference Series in Mathematics}.
\newblock American Mathematical Society, Providence, RI, 1994.

\bibitem[MR14]{MatolcsiRuzsa2012GeneralProperties}
{M.} Matolcsi and {I.~Z.} Ruzsa.
\newblock Difference sets and positive exponential sums {I}. general properties.
\newblock {\em Journal of Fourier Analysis and Applications}, 20(1):17--41, 2014.

\bibitem[MR21]{MatolcsiRuzsa2021CubicResidues}
{M.} Matolcsi and {I.~Z.} Ruzsa.
\newblock Difference sets and positive exponential sums {II}: Cubic residues in cyclic groups.
\newblock {\em Proceedings of the Steklov Institute of Mathematics}, 314(1):138--143, 2021.

\bibitem[NS14]{NincevicSlijepcevic2013}
M.~Nin{\v{c}}evi{\'c} and S.~Slijep{\v{c}}evi{\'c}.
\newblock Positive exponential sums and odd polynomials.
\newblock {\em Rad Hrvatske akademije znanosti i umjetnosti. Matemati{\v{c}}ke znanosti}, 18(519):35--53, 2014.

\bibitem[Ric20]{RiceBinaryQuadraticForms}
A.~Rice.
\newblock Binary quadratic forms in difference sets.
\newblock In M.~B. Nathanson, editor, {\em Combinatorial and Additive Number Theory III}, volume 297 of {\em Springer Proceedings in Mathematics \& Statistics}, pages 175--196. Springer, Cham, 2020.

\bibitem[Ruz84]{RuzsaConnections}
I.~Z. Ruzsa.
\newblock Connections between the uniform distribution of a sequence and its differences.
\newblock In {\em Topics in Classical Number Theory, Vol. I, II}, volume~34 of {\em Colloquia Mathematica Societatis J{\'a}nos Bolyai}, pages 1419--1443. North-Holland, Amsterdam, 1984.

\bibitem[Sli10]{Slijepcevic2010Squares}
S.~Slijep{\v{c}}evi{\'c}.
\newblock On van der {C}orput property of squares.
\newblock {\em Glasnik Matemati{\v{c}}ki. Serija III}, 45(2):357--372, 2010.

\bibitem[You19]{YounisLowerBounds}
K.~Younis.
\newblock Lower bounds in the polynomial {S}zemer{\'e}di theorem, 2019.
\newblock arXiv:1908.06058.

\end{thebibliography}

\end{document}